\documentclass[usenames,dvipsnames, 11pt]{article}

\usepackage[utf8]{inputenc}
\usepackage[T1]{fontenc}

\usepackage[margin=2.5cm,a4paper]{geometry} 
\usepackage{multirow,listings,setspace,latexsym,keyval,ifthen,moreverb,lscape,forest}
\usepackage{pgf,tikz}
\usetikzlibrary{arrows}

\usepackage{amsmath, amssymb, amsthm, bm}
\usepackage{thm-restate}
\usepackage{hyperref}
\usepackage[capitalize]{cleveref}
\usepackage{color}
\usepackage{url}
\usepackage{enumitem}
\usepackage{graphicx}
\usepackage{appendix}
\usepackage{cite}
\usepackage{mathtools}

\newtheorem{thm}{Theorem}

\newtheorem{lem}[thm]{Lemma}
\newtheorem{cor}[thm]{Corollary}

\numberwithin{thm}{section}

\theoremstyle{definition}
\newtheorem{mydef}{Definition}

\theoremstyle{definition}
\newtheorem{rem}{Remark}

\theoremstyle{definition}
\newtheorem{conj}{Conjecture}

\newcommand{\cov}[1]{\operatorname{cov}_{#1}}

\usepackage{diagbox}

\def\cL{{\mathcal L}}
\def\cH{{\mathcal H}}
\def\cC{{\mathcal{C}}}
\newcommand{\R}{\mathbb{R}}

\title{Covering half-grids with lines and planes}
\author{Anurag Bishnoi\thanks{Delft Institute of Applied Mathematics, TU Delft.
Email: \texttt{A.Bishnoi@tudelft.nl}} \,and 
Shantanu Nene}
\begin{document}

\maketitle

\begin{abstract}
    We study hyperplane covering problems for finite grid-like structures in $\mathbb{R}^d$.
    We call a set $\mathcal{C}$ of points in $\mathbb{R}^2$ a conical grid if the line $y = a_i$ intersects $\mathcal{C}$ in exactly $i$ points, for some $a_1 > \cdots > a_n \in \mathbb{R}$. 
    We prove that the number of lines required to cover every point of such a grid at least $k$ times is at least $nk\left(1-\frac{1}{e}-O(\frac{1}{n}) \right)$. 
    If the grid $\mathcal{C}$ is obtained by cutting an $m \times n$ grid of points in half along one of the diagonals, then we prove the lower bound of $mk\left(1-e^{-\frac{n}{m}}-O(\frac{n}{m^2})\right)$.
    
    In general, we call a grid obtained by cutting a grid in $\mathbb{R}^d$ along one of the diagonals a \textit{half-grid}.
    Motivated by the Alon-F\"uredi theorem on hyperplane coverings of grids that miss a point and its multiplicity variations, we study the problem of finding the minimum number of hyperplanes required to cover every point of an $n \times \cdots \times n$ half-grid in $\mathbb{R}^d$ at least $k$ times while missing a point $P$.
   For almost all such half-grids, with $P$ being the corner point, we prove asymptotically sharp upper and lower bounds for the covering number in dimensions $2$ and $3$.
   For $k = 1$, $d = 2$, and an arbitrary $P$, we determine this number exactly by using the polynomial method bound for grids.
\end{abstract}

\section{Introduction}
Covering problems for finite subsets of $\mathbb{R}^d$ have a long history and various connections to algebra, topology, and extremal combinatorics. 
For example, the classical Cayley-Bacharach theorem \cite{eisenbud1996cayley} states that if a set $\Gamma$ of $9$ points in $\mathbb{R}^2$ can be obtained as an intersection of two cubic curves then any cubic curve that passes through $8$ out of these $9$ points must also pass through the remaining point. 
In other words, the degree of any polynomial in $\mathbb{R}[x, y]$ that vanishes on all points of $\Gamma$ except one is at least $4$. 
A special case of this theorem is when the two cubic curves are of the form $(x-a_0)(x - a_1)(x - a_2)$ and $(y - b_0)(y - b_1)(y - b_2)$, that is, the set $\Gamma$ is equal to $\{a_0, a_1, a_2\} \times \{b_0, b_1, b_2\}$, and we want to cover all points using lines while missing the point $(a_0, b_0)$. 
The Cayley-Bacharach theorem then implies that any such cover must contain at least $4$ lines.  
Motivated by a Ramsey theoretic problem \cite{K94}, Alon and F\"uredi \cite{AF93} proved a generalization of the above observation for every grid $\Gamma = S_1 \times \cdots \times S_d \in \mathbb{F}^d$, where $\mathbb{F}$ is an arbitrary field and $S_1, \dots, S_d \subseteq \mathbb{F}$ are finite sets.
They showed that the minimum number of hyperplanes required to cover all points of $\Gamma$ while missing one is equal to $\sum_{i = 1}^d (|S_i| - 1)$. 
Note that if we do not have a requirement of missing one point then the minimum is easily seen to be equal to $\min_{i} |S_i|$, which is much smaller. 
We also note that in dimensions greater than $2$, this result of Alon and F\"uredi follows from a more general version of the Cayley-Bacharach theorem that requires modern algebraic geometric tools; see, for example, \cite[Thm. CB7]{eisenbud1996cayley}, \cite[Thm. 5.1]{karasev2019residues} or \cite[Sec.~2.3]{Verlinde24} for a more detailed discussion.

The work of Alon and F\"uredi has played a significant role in the development of the \textit{polynomial method} in combinatorics \cite{Alon99, BGGSZ22}.
Therefore, various generalizations of this result have been studied in the literature. 
One such generalization is the \textit{multiplicity version} where we require the set of hyperplanes to cover each point at least $k$ times, while missing one point \cite{BS09}. 
Motivated by results in finite geometry that precede the Alon-F\"uredi theorem \cite{J77, BS78, B92}, Ball and Serra proved that one needs at least $k(|S_1| - 1) + (|S_2| - 1) + \cdots + (|S_n| - 1)$ hyperplanes, where $S_1$ corresponds to the largest side of the grid. 
Although this bound is tight when $|S_1|$ is much larger than $|S_i|$ for all $i \neq 2$ (see the remark after the proof of Theorem 1.1 in \cite{BBDD23}), in some combinatorially interesting special cases, like the binary grid $\{0, 1\}^n \subset \mathbb{R}^n$, it is an open problem to determine the correct bound \cite{CH20, SW20}. 

Another generalization is to cover other finite sets of points using hyperplanes \cite{BBSz10, basit2023covering}, as done by the Cayley-Bacharach theorem where the points are intersections of arbitrary cubic curves. 
Taking inspiration from this, and in particular the recent work of Basit, Clifton, and Horn on triangular grids \cite{basit2023covering}, we study the covering problem for the following finite subsets of points. 

\begin{mydef}[Conical grids]
    For a positive integer $n$ and real numbers $a_0 < a_1< \cdots < a_{n-1}$, a conical grid on $S = \{a_0, a_1, \dots, a_{n-1}\}$ is a subset $\cC$ of $\R^2$ of size $\frac{n(n+1)}{2}$ such that the line $y=a_i$ contains exactly $n-i$ points of $\cC$. The number $n$ is called the order of $\cC$.
\end{mydef}

\begin{mydef}[Half-grids]
    Let $S_1, S_2$ be two subsets of $\R$ of sizes $n$ and $m$ respectively, where $m \leq n$. Suppose $S_1=\{a_0, a_1, \cdots,  a_{n-1}\}$ and $S_2=\{b_0, b_1, \cdots, b_{m-1}\}$, with $a_0 < a_1 < \cdots < a_{n-1}$ and $b_0 < b_1 < \cdots < b_{m-1}$. 
    Then the half-grid $\mathcal{H}$ defined by $S_1$ and $S_2$ is the set of points $(a_i,b_j)$ such that $(m-1)i+(n-1)j \leq (m-1)(n-1)$. 
    The \textit{vertex} of a half-grid $\cH$ is the point $(a_0,b_0)$. The lines $x=a_0$ and $y=b_0$ together form the \textit{boundary} of $\cH$, and all other points are interior points.
\end{mydef}

Note that an $n \times n$ half-grid is a special kind of conical grid, and it is easy to see, by repeating $k$ times each line of the form $y=a_i$, that all points of a conical grid of order $n$ can be covered $k$ times using $nk$ lines. 
We prove the following general lower bound. 

\begin{thm}
\label{thm:conical_mult}
Let $\Gamma$ be a conical grid of order $n$. The minimum number of lines required to cover every point in $\Gamma$ at least $k$ times is at least $nk\left(1-\frac{1}{e}-O(\frac{1}{n}) \right)$.
\end{thm}

For the $m \times n$ half-grids, we prove the following lower bound which is equal to the previous bound when $m = n$. 

\begin{thm}
\label{thm:half_mult}
    Let $\cH$ be an $m \times n$ half-grid, with $2 \leq m \leq n$. Then the minimum number of lines required to cover every point of $\cH$ at least $k$ times is at least $mk\left(1-e^{-\frac{n}{m}}-O(\frac{n}{m^2})\right)$.
\end{thm}

As a side note, any time $O()$ or $\Theta()$ notation is used in this paper, it indicates constants independent of all parameters in the expression. For instance, $a \leq b+O(c)$ implies that there exists a constant $M$, independent of $b$ and $c$, such that $a \leq b+Mc$ for all $b,c$.

In \cite{basit2023covering}, Basit, Clifton, and Horn have found a lower bound for $k$-covering of $n \times n$ \textit{equally spaced} half-grids (which they called triangular grids), that comes out to be $\frac{2nk}{3}-O(k)$.
Their bound relied on the particular structure of the half-grid, whereas our slightly weaker bounds apply to a much more general arrangement of points.

We now turn to the case of missing one point, as done by Alon and F\"uredi. 
Here, the answer depends on the point that is being missed and the structure of the grids. 
Using the result of Alon-F\"uredi for grids, we prove the following for half-rectangular grids, which are defined in Section~\ref{sec:structured}.

\begin{thm}
\label{thm:structured}
    For an $m \times n$ half-rectangular grid $\cH$, with $2 \leq m \leq n$, and a point $P=(x_0,y_0) \in \cH$, the minimum number of lines required to cover each point of $\cH$, while missing $P$, is equal to $n-\lceil \frac{n-m}{m-1} y_0 \rceil -1$.
\end{thm}

We now move to covering with multiplicity. 
For simplicity, we assume that the point that is missed is the vertex of the half-grid, that is, the point where the horizontal and the vertical axes meet. 
In this scenario, the minimum number of hyperplanes required to cover the equally spaced $n \times \dots \times n$ grid $\Gamma = \{(x_1, \dots, x_n) : \forall i, x_i \in \mathbb{Z}^{\geq 0}, \sum x_i \leq n-1\}$, while missing the origin, 
is seen to be $(n-1)k$. The lower bound follows because no two of the $n-1$ non-vertex points along the same axis can be covered by the same hyperplane, and for an upper bound, consider each of the hyperplanes $\sum x_i = j$, for $1 \leq j \leq n-1$ repeated $k$ times.
Thus, we focus on the case of a `generic' half-grid (see Section~\ref{sec:missing_point} for a definition), where such a small cover is impossible and determining the correct bound is highly nontrivial. 
Using linear programming methods, we are able to prove the following asymptotically tight bounds in dimensions $2$ and $3$. 

\begin{thm}
\label{thm:generic_dim2}
    For a generic $n \times n$ half-grid in $\mathbb{R}^2$, the minimum number of lines required to cover every point at least $k$ times while missing the vertex is at least
    $\frac{3nk}{2} - 2k$ and at most $\frac{3nk}{2} + \frac{k}{2}$.
\end{thm}

   \begin{thm}
   \label{thm:generic_dim3}
    For any 3-dimensional generic half-grid $\Gamma$ of order $n+1$, 
    \[
\frac{31}{18}nk-O(k)
\leq
\cov{k}(\Gamma)
\leq
\frac{31}{18}nk+O(n^2+k),
\]
\end{thm}

\begin{rem}
    Note that for $1 \ll n \ll k$, both upper and lower bounds converge to $31nk/18$.
    If $n \gg k$, then the lower bound is negative, which is trivially true.
\end{rem}

\section{Covering grids with multiplicity}
In this section, we prove Theorems~\ref{thm:conical_mult} and \ref{thm:half_mult}. 

\begin{lem}
    Let $\alpha_1>\alpha_2> \cdots >\alpha_n$ be a sequence of real numbers and $\beta_1 \leq \beta_2 \leq \cdots \leq \beta_n$ a sequence of positive integers. Let $S$ be a set of $\beta_1+\beta_2+\cdots +\beta_n$ points in $\mathbb{R}^2$ such that the line $y=\alpha_i$ contains exactly $\beta_i$ points of $S$. Let
    $\mathcal{C}$ be a multiset of lines that covers every point of $S$ at least $k$ times. Say $\mathcal{C}$ has $l_0$ non-horizontal lines, and let $t$ be a positive integer such that $\beta_t \geq \frac{l_0}{k}$, then the size of $\mathcal{C}$ is at least
    $$(n-t+1)k+l_0 \left(1-\sum_{i=t}^n \frac{1}{\beta_i} \right).$$
\end{lem}
\begin{proof}
    Suppose we use $l_i$ lines of the form $y=\alpha_i$. Then the total number of points covered on $y=\alpha_i$ (with multiplicity) is at most $l_0+\beta_il_i$, and this must be at least $k\beta_i$. Hence $l_i \geq k- \frac{l_0}{\beta_i}$ for all $i$. 
    After deleting horizontal lines that contain no point of $S$, we may assume
that every horizontal line in $\mathcal C$ is one of the lines
$y=\alpha_i$. Thus the number of lines used is
    $$l_0+\sum_{i=1}^n l_i \geq l_0 + \sum_{i = t}^n l_i \geq (n-t+1)k+l_0 \left(1-\sum_{i=t}^n \frac{1}{\beta_i} \right),$$
    as required.
\end{proof}

We apply the above lemma to conical and half-grids.

\begin{cor}
    Let $\Gamma$ be a conical grid of order $n$. The minimum number of lines required to cover every point in $\Gamma$ at least $k$ times is at least $nk\left(1-e^{\frac{1}{2n}-1}-\frac{1}{n}\right) = nk\left(1-\frac{1}{e}-O(\frac{1}{n}) \right)$.
\end{cor}
\begin{proof}
    Suppose the conical grid is on the set $\{a_0<a_1< \cdots < a_{n-1} \}$. Then, using the notation of the lemma, $\alpha_i = a_{n-i}$. Suppose we use $l_i$ lines of the form $y = \alpha_i$ and $l_0$ non-horizontal lines. 
    In this case, $\beta_i=i$, and so the smallest $s$ such that $s=\beta_s \geq \frac{l_0}{k}$ is $s=\lceil \frac{l_0}{k} \rceil$. 
    If $s=0$, then $l_0=0$, and the lemma with $t=1$ gives a lower bound
of $nk$, which is stronger than the desired bound. Thus we may assume
$s\geq 1$.
    We use $t=s+1$ in the lemma. If $t>n$, that would imply $l_0>(n-1)k$ so we are already done. So we assume $t \leq n$ so $\beta_t$ is well defined.
    Writing $l=l_0$, the total number of lines required is at least 
    $$ l+(n-s)k-l\left(\frac{1}{s+1}+\cdots+\frac{1}{n}\right) = (n-s)k+l-l(H_n-H_{s})>(n-s)k+l-l(H_n-\ln s-\gamma)$$
    where $\gamma$ is the Euler-Mascheroni constant and $H_j$ denotes the $j$-th partial sum of the harmonic series. In the last step, we have used the bound $H_{s}>\ln s +\gamma$. Now we have two cases: 
    
    \textbf{Case I:} $1+ \gamma + \ln s - H_n \leq 0$. In this case, we use $l \leq sk$ to deduce that the number of lines is $>nk-sk(H_n-\ln s-\gamma)$. 
    The minimum of the above expression occurs, after taking the derivative, when $s=e^{H_n-\gamma-1}$. At this point, the expression becomes $k(n-s)$, and using the inequality $H_n < \ln n + \gamma + \frac{1}{2n}$, we find that the number of lines needed is greater than $nk(1-e^{\frac{1}{2n}-1})$.

    \textbf{Case II:} $1+ \gamma + \ln s - H_n > 0$. In this case, we use the bound $l > sk-k$ to get that the number of lines is $>nk-sk(H_n-\ln s-\gamma)-k(1+\gamma + \ln s - H_n)>nk-sk(H_n-\ln s-\gamma)-k$, since $s \leq n$ implies $H_n \geq H_s > \ln s +\gamma$. This is just $k$ less than the bound in the previous case, so in this case we get a lower bound of $nk\left(1-e^{\frac{1}{2n}-1}-\frac{1}{n}\right)$.
\end{proof}

\begin{cor}
    Let $\cH$ be an $m \times n$ half-grid, with $m \leq n$. Then the minimum number of lines required to cover every point of $\cH$ at least $k$ times is at least $mk\left(1-e^{-\frac{n}{m}}-O(\frac{n}{m^2})\right)$.
\end{cor}
\begin{proof}
    Note that in the case of $n=m$, $\cH$ is a conical grid, and the bound reduces to the same one as the previous case. Hence from now on, we assume $n \geq m+1$. 
    
    Let the sets of $\cH$ be $S_1=\{c_0<c_1<\cdots<c_{n-1}\}$ and $S_2=\{d_0<d_1<\cdots<d_{m-1}\}$. Then, using the notation of the lemma, $\alpha_i=d_{m-i}$ for $1 \leq i \leq m$, and 
    $$\beta_i=1+\left\lfloor\frac{(i-1)(n-1)}{m-1}\right\rfloor > \frac{(i-1)(n-1)}{m-1}.$$ 
    Let $l=l_0$ for convenience. If $l \geq mk$, then we are already done, so we may assume $l < mk$.
 
    Define $s = \lceil \frac{l(m-1)}{k(n-1)} \rceil$ and $t = 1+s$. 
    Since $l < mk$, we have $s \leq \lceil \frac{m(m-1)}{n-1} \rceil \leq m-1$, so $t \leq m$ and $\beta_t$ is well-defined.
    Moreover, $\frac{(t-1)(n-1)}{m-1} = \frac{s(n-1)}{m-1} \geq \frac{l}{k}$, so the condition $\beta_t \geq l/k$ of the lemma is satisfied.
 
    By the lemma, the number of lines used is at least
    $$l+(m-s)k - l\left(\frac{1}{\beta_{s+1}}+\cdots + \frac{1}{\beta_m}\right).$$
    We first handle the case $s=0$. This can only happen if $l=0$, and in that case the above bound becomes $mk$, so we are done. \\
    
    Since $\beta_i > \frac{(i-1)(n-1)}{m-1}$, we have $\frac{1}{\beta_i} < \frac{m-1}{(i-1)(n-1)}$, giving
    \begin{align}
    l+(m-s)k-l\left(\frac{1}{\beta_{s+1}}+\cdots + \frac{1}{\beta_m}\right) & > l+(m-s)k-\frac{l(m-1)}{n-1} \left( \frac{1}{s}+\frac{1}{s+1}+\cdots+\frac{1}{m-1}\right) \notag\\
    &= (m-s)k+l\left(1-\frac{(m-1)(H_{m-1}-H_{s-1})}{n-1}\right). \label{eq:main_expr}
    \end{align}

    Next suppose $s=1$. Then \eqref{eq:main_expr} gives the lower bound
    \[
    (m-1)k+
    l\left(1-\frac{(m-1)H_{m-1}}{n-1}\right).
    \]
    If the coefficient of $l$ is non-negative, this is at least
    $(m-1)k=mk(1-1/m)$, which is enough since $n\geq m$ and hence
    $1/m=O(n/m^2)$.

    It remains to consider the case where the coefficient of $l$ is negative.
    Then $n-1<(m-1)H_{m-1}$. Since $s=1$, we have
    $l\leq k(n-1)/(m-1)$, and therefore
    \begin{align*}
    (m-1)k+
    l\left(1-\frac{(m-1)H_{m-1}}{n-1}\right)
    &\geq
    (m-1)k+
    \frac{k(n-1)}{m-1}
    \left(1-\frac{(m-1)H_{m-1}}{n-1}\right)  \\
    &=
    (m-1)k+\frac{k(n-1)}{m-1}-kH_{m-1}
    \geq
    mk-kH_{m-1}.
    \end{align*}
    We now use the elementary estimate
    $H_{m-1}/m\leq e^{-n/m}+O(n/m^2)$. Indeed, if
    $n/m\leq \frac12\log m$, then $e^{-n/m}\geq m^{-1/2}$ dominates
    $H_{m-1}/m=O(\log m/m)$; if $n/m>\frac12\log m$, then
    $n/m^2\gg \log m/m$. Hence
    \[
    mk-kH_{m-1}
    =
    mk\left(1-\frac{H_{m-1}}{m}\right)
    \geq
    mk\left(1-e^{-n/m}-O\left(\frac{n}{m^2}\right)\right).
    \]
    
        For the remainder of the proof, assume $s\geq 2$. Put
    $A=\frac{n-1}{m-1}$ and
    \[
    D(s)=H_{m-1}-\ln(s-1)-\gamma.
    \]
    Since $H_{s-1}>\ln(s-1)+\gamma$, \eqref{eq:main_expr} is greater than
    \begin{equation}\label{eq:f_bound}
    (m-s)k+l\left(1-\frac{D(s)}{A}\right).
    \end{equation}
    We repeatedly use
    \[
    \frac{m-1}{m}e^{\frac{1}{2(m-1)}-\frac{n-1}{m-1}}
    =
    e^{-n/m+O(n/m^2)},
    \]
    and the fact that $O(k)$ is absorbed into $mk\cdot O(n/m^2)$ since
    $n\geq m$.

    \textbf{Case I:} $D(s)\leq A$. The coefficient of $l$ in
    \eqref{eq:f_bound} is non-negative, so using
    $l>\frac{(s-1)k(n-1)}{m-1}=(s-1)kA$, the expression is greater than
    \[
    f(s):=(m-s)k+(s-1)k(A-D(s)).
    \]
    Since $f'(s)=k(A-D(s))\geq 0$ throughout this case, $f$ is minimized at
    the left boundary $s_0=1+e^{H_{m-1}-\gamma-A}$, where $D(s_0)=A$. Hence
    \[
    f(s)\geq f(s_0)=mk-s_0k.
    \]
    Using $H_{m-1}<\ln(m-1)+\gamma+\frac{1}{2(m-1)}$, we get
    \[
    s_0<1+(m-1)e^{\frac{1}{2(m-1)}-\frac{n-1}{m-1}}.
    \]
    Therefore
    \[
    f(s)
    \geq
    mk-k-k(m-1)e^{\frac{1}{2(m-1)}-\frac{n-1}{m-1}}
    \geq
    mk\left(1-e^{-n/m}-O\left(\frac{n}{m^2}\right)\right).
    \]

    \textbf{Case II:} $D(s)>A$. The coefficient of $l$ in
    \eqref{eq:f_bound} is negative, so using
    $l\leq \frac{sk(n-1)}{m-1}=skA$, the expression is at least
    \[
    h(s):=(m-s)k+sk(A-D(s))
    =
    mk+sk(A-1-D(s)).
    \]
    We have $h''(s)=\frac{k(s-2)}{(s-1)^2}\geq 0$, so it suffices to check
    the possible boundary points and the possible interior critical point.

    If $s^*$ is an interior critical point, then
    $D(s^*)=A+\frac{1}{s^*-1}$, and hence
    \[
    h(s^*)=
    mk-\frac{(s^*)^2k}{s^*-1}
    =
    mk-k(s^*-1)-2k-\frac{k}{s^*-1}.
    \]
    Since
    \[
    s^*-1<e^{H_{m-1}-\gamma-A}
    <
    (m-1)e^{\frac{1}{2(m-1)}-\frac{n-1}{m-1}},
    \]
    we get
    \[
    h(s^*)\geq
    mk-k(m-1)e^{\frac{1}{2(m-1)}-\frac{n-1}{m-1}}-O(k)
    \geq
    mk\left(1-e^{-n/m}-O\left(\frac{n}{m^2}\right)\right).
    \]

    At the boundary $s=s_0$, if it occurs, we have $D(s_0)=A$, so
    $h(s_0)=mk-s_0k$. The same estimate as in Case~I gives
    \[
    h(s_0)\geq
    mk\left(1-e^{-n/m}-O\left(\frac{n}{m^2}\right)\right).
    \]

    At the endpoint $s=2$, if it occurs, then
    $A\leq H_{m-1}-\gamma=O(\log m)$, so $n/m=O(\log m)$. Also
    \[
    h(2)=
    mk+
    2k\left(\frac{n-m}{m-1}-(H_{m-1}-\gamma)\right)
    \geq
    mk-O(k\log m).
    \]
    Since, in this regime,
    $\log m/m\leq e^{-n/m}+O(n/m^2)$, this is enough.

    Finally, at the endpoint $s=m-1$, if it occurs, then
    \[
    h(m-1)
    =
    nk-(m-1)k\left(H_{m-1}-\ln(m-2)-\gamma\right)
    =
    nk-O(k).
    \]
    This is at least $mk-O(k)$, and hence again at least
    \[
    mk\left(1-e^{-n/m}-O\left(\frac{n}{m^2}\right)\right).
    \]

    Thus in every case the number of lines is at least
    \[
    mk\left(1-e^{-n/m}-O\left(\frac{n}{m^2}\right)\right),
    \]
    as required.
\end{proof}

Note that the above bound is non-trivial only in the regime $n=o(m^2)$. In fact, when $n-1\geq (m-1)H_{m-1}$, we get a much stronger bound:

\begin{cor}
    Let $\cH$ be an $m \times n$ half-grid, with
    \[
    n-1 \geq (m-1)H_{m-1}.
    \]
    Then the minimum number of lines required to cover every point of $\cH$
    at least $k$ times is at least $(m-2)k$.
\end{cor}

\begin{proof}
    We follow the proof of the previous corollary up to
    \eqref{eq:main_expr}. If $m\leq 2$, the bound is trivial, so assume
    $m\geq 3$.

    By assumption,
    \[
    \frac{n-1}{m-1}\geq H_{m-1}.
    \]
    Hence, for every $s\geq 1$,
    \[
    1-\frac{(m-1)(H_{m-1}-H_{s-1})}{n-1}\geq 0.
    \]
    Thus the coefficient of $l$ in \eqref{eq:main_expr} is always
    non-negative.

    If $s=0$, the bound from the previous proof is $mk$, which is stronger.
    If $s=1$, \eqref{eq:main_expr} gives at least $(m-1)k$, which is again
    stronger than $(m-2)k$.

    Now assume $s\geq 2$. As in the previous proof, using
    $H_{s-1}>\ln(s-1)+\gamma$ gives \eqref{eq:f_bound}. Since
    \[
    H_{m-1}-\ln(s-1)-\gamma < H_{m-1}\leq \frac{n-1}{m-1},
    \]
    the coefficient of $l$ in \eqref{eq:f_bound} is also non-negative.
    Therefore, using
    \[
    l>\frac{(s-1)k(n-1)}{m-1},
    \]
    the expression in \eqref{eq:f_bound} is greater than
    \[
    f(s):=(m-s)k+
    \frac{(s-1)k(n-1)}{m-1}
    \left(
    1-\frac{(m-1)(H_{m-1}-\ln(s-1)-\gamma)}{n-1}
    \right).
    \]
    As before,
    \[
    f'(s)=
    k\left(
    \frac{n-1}{m-1}
    -
    (H_{m-1}-\ln(s-1)-\gamma)
    \right)>0
    \]
    for all $s\geq 2$. Hence $f(s)$ is minimized at $s=2$. Thus
    \[
    f(s)\geq f(2)
    =
    (m-2)k+
    k\frac{n-1}{m-1}
    \left(
    1-\frac{(m-1)(H_{m-1}-\gamma)}{n-1}
    \right)
    \geq
    (m-2)k,
    \]
    where the last inequality follows again from
    $\frac{n-1}{m-1}\geq H_{m-1}$. This proves the claim.
\end{proof}
We do not believe that the lower bound in the conical grid case is optimal.  In the equally
spaced case, that is, for triangular grids, the fractional lower bound of
Basit, Clifton, and Horn gives a lower bound of the form
\[
\frac{2nk}{3}-O(k).
\]
Moreover, for fixed $n$ and $k\to\infty$, the integral covering number is
asymptotic to the corresponding fractional covering number. This suggests
that the same lower bound should hold for arbitrary conical grids.

\begin{conj}
    There exists an absolute constant $C>0$ such that the following holds.
    Let $\Gamma$ be a conical grid of order $n$. Then every multiset of lines
    that covers every point of $\Gamma$ at least $k$ times has size at least
    \[
    \frac{2nk}{3}-Ck.
    \]
\end{conj}

\section{Covering with multiplicity without one point}
\label{sec:missing_point}

Here we study problems of covering half-grids while missing one of the points.
If we cover an $n \times n \times \cdots \times n$ equally-spaced half-grid, then the optimal answer is $k(n-1)$. 
Thus, we focus on arbitrary half-grids and prove Theorems~\ref{thm:generic_dim2} and \ref{thm:generic_dim3}.

\subsection{Generic grids in dimension $2$}

\begin{mydef}
    For a half-grid $\cH$ and a point $P \in \cH$, define $\cov{k}(\cH,P)$ to be the minimum number of lines such that every point in $\cH \setminus P$ lies on at least $k$ of the lines and none of the lines pass through $P$. We also use the notation $\cov{k}(\cH)=\cov{k}(\cH,V)$ where $V$ is the vertex of $\cH$.
\end{mydef}

We study the covering number of \textit{generic half-grids}, as defined below, which are similar to, but slightly stronger than the generic grids introduced in \cite{BBDD23}. 

\begin{mydef}
    A half-grid $\cH$ on $S_1$ and $S_2$ is called generic if every line that is not parallel to the $x$-axis or $y$-axis passes through at most $2$ points of $\cH$.
\end{mydef}

\begin{rem}
    It can be easily verified that if the sets $S_1$ and $S_2$ are picked randomly, say from the uniform distribution on the interval $[0, 1]$, then with probability $1$ the half-grid will have the property above and thus justify the use of the term `generic'. 
\end{rem}

We will be using the linear programming method to find lower bounds for coverings, as done in \cite{BBDD23}. A covering can be thought of as an integer linear program (ILP) to minimize some function, and so the solution to it would be lower-bounded by the solution to the corresponding real LP. 
Furthermore, by duality, this solution would be lower-bounded by any solution to the dual LP. The dual LP for the covering problem would be to give weights to points so that the sum of the weights of points on any line is at most $k$ (for $k$-times covering), and we produce lower bounds by giving such weights.
To be more precise, 
$\cov{k}(\Gamma)$, for a half-grid $\Gamma$ in $\mathbb{R}^2$ is the solution to the following ILP, where $\mathcal{L}$ is the set of lines that intersect $\Gamma$ in at least two points:

\begin{align*}
    &\text{minimize} \quad \sum\limits_{\ell\in \cL} z(\ell)\\
    &\text{subject to}\notag\\
    &\quad\quad\quad\quad\sum\limits_{\substack{\ell\in \cL:\\ (x,y)\in \ell}} z(\ell) \geq k \quad \text{for all }(x,y)\in \Gamma \setminus \{(0,0)\} \\
    &\quad\quad\quad\quad z(\ell) \in \mathbb{Z}_{\geq 0} \quad \text{for all }\ell\in \cL
\end{align*}

\noindent
The dual of the corresponding LP-relaxation of the ILP is given by
\begin{align*}
    &\text{maximize} \quad \sum\limits_{(x,y)\in \Gamma \setminus \{(0,0)\}} w(x,y)\\
    &\text{subject to}\notag\\
    &\quad\quad\quad\quad \sum\limits_{\substack{(x,y)\in \Gamma \setminus \{(0,0)\}:\\ (x,y)\in \ell}} w(x,y) \leq k \quad \text{for all }\ell\in \mathcal{L}\notag\\
    &\quad\quad\quad\quad w(x,y) \geq 0 \quad \text{for all }(x,y)\in \Gamma \setminus \{(0,0)\} \notag
\end{align*}

Therefore, we just need to find a feasible solution to this dual linear program to get a lower bound on the primal.
In particular, any assignment of real weights $w(x, y)$ with the constraints above, gives us \[\cov{k}(\Gamma) \geq \sum_{(x, y) \in \Gamma \setminus \{(0, 0)\}} w(x, y).\]
We refer the reader to \cite[Section 2]{BBDD23} for a more detailed discussion on this approach. 

We now prove lower and upper bounds on the covering number of generic square half-grids in $\mathbb{R}^2$ that determine this function asymptotically for fixed $n \geq 4$ and $k \rightarrow \infty$.

\begin{thm}
    For a generic half-grid $\cH$ of order $n \geq 4$, $$\frac{3nk}{2}+\frac{k}{2} \geq \cov{k}(\cH) \geq \frac{3nk}{2}- 2k.$$
\end{thm}
\begin{proof}
Let $\mathcal{H}$ be given by the points $a_0, a_1, \dots, a_{n -1}$ on the $x$-axis and the points $b_0, b_1, \dots, b_{n - 1}$ on the $y$-axis, with $V = (a_0, b_0)$ as its vertex.
 We first prove the upper bound via a construction.
Choose $\lceil \frac{n}{2} \rceil \leq \frac{n+1}{2}$ lines of the form $y=b_i$ for $1 \leq i \leq \lceil \frac{n}{2} \rceil$ and $\lceil \frac{n}{2} \rceil \leq \frac{n+1}{2}$ lines of the form $x=a_i$ for $1 \leq i \leq \lceil \frac{n}{2} \rceil$. These cover all points of $\cH \setminus V$ except the $2\lfloor \frac{n}{2} \rfloor-2$ points given by $(a_i, b_0)$ and $(a_0, b_i)$ with $i > \lceil n/2 \rceil$. 
We pair these points up and cover them with $\lfloor\frac{n}{2} \rfloor-1$ more lines. Now we repeat all of these lines $k$ times. 
We have used at most $k(2\lceil \frac{n}{2} \rceil+\lfloor \frac{n}{2} \rfloor) \leq \frac{3nk}{2}+\frac{k}{2}$ lines.

For the lower bound, 
we use linear programming. In particular, we consider the dual of the linear relaxation of the covering problem (see \cite[Section 2]{BBDD23}): we have to give weights to every point except $V$ such that the sum of the weights on any line not passing through $V$ is at most $k$, and we have to maximize the sum of the weights. Any such valid weighting will give a lower bound for our primal LP, which in turn would be a lower bound for the number of lines. 

    Give weight $\frac{k}{2}$ to points on lines $x=a_0$ and $y=b_0$, and also to points of the form $(a_i,b_{n-1-i})$ for $1 \leq i \leq n - 2$. Give all other points a weight of $0$. 
    This is a valid weighting due to the grid being generic; indeed, any line passes through at most two points that have non-zero weight, and these weights sum to $k$. 
    The total sum of the weights is $\frac{3nk}{2}-2k$, as required. 
\end{proof}

\begin{rem}
    Note that we get the same asymptotic bounds for the generic full grid (see \cite[Theorem 1.5]{BBDD23}), but the construction and the weighting are completely different. 
\end{rem}

\begin{rem}
    The only place where we use genericity of the grid is the fact that no three points on the entire boundary  $(a_i,b_{n-i})$, $x=a_0$, or $y=b_0$ are collinear, and thus our lower bound holds for any half-grid satisfying this weaker property.
\end{rem}

\subsection{Generic grids in dimension $3$}

We extend the definition of generic to higher-dimensional grids as follows. A grid $S_1 \times S_2 \times S_3$ in $\mathbb{R}^3$ is considered generic if every non-empty planar slice, obtained by fixing one coordinate and allowing the others to vary, forms a generic grid in $\mathbb{R}^2$ \footnote{Here we mean that any line not parallel to the axes should intersect the grid in at most $2$ points.}, and any plane not parallel to any axis contain at most $3$ points of the grid. 
Generic half-grids can be obtained by cutting such a generic grid in half. 
We formalize this in the following definition, where $\mathrm{supp}(\vec{u})$ denotes the set of coordinate positions where $\vec{u}$ has a nonzero entry.
Here we only focus on the case where all $S_i$'s have the same size.

\begin{mydef}[Generic half-grid in 3 dimensions]
    Let $S_1=\{0=a_0<a_1<\cdots<a_n\}$, $S_2=\{0=b_0<b_1<\cdots<b_n\}$, $S_3=\{0=c_0<c_1<\cdots<c_n\}$ be three sets of real numbers. Let $\Gamma=\{(a_i,b_j,c_k) \mid i+j+k \leq n\}$. Then $\Gamma$ is called a 3-dimensional half-grid of order $n+1$. 
    $\Gamma$ is called generic if for any plane $\vec{u}^T \cdot \vec{x} = 1$, with $\mathrm{supp}(\vec{u}) = I \subseteq [3]$, and any scalars $\lambda_j$, with $j \in [3] \setminus I$, we have 
    \[\left| (x_1, x_2, x_3) \in \Gamma : \sum u_i x_i = 1, x_j = \lambda_j~\forall j \in [3] \setminus I \right| \leq |I|.\]
\end{mydef}

\begin{mydef}
    For a 3-dimensional generic half-grid $\Gamma$, define $\cov{k}(\Gamma)$ to be the minimum number of planes such that every point in $\Gamma \setminus \{(0,0,0)\}$ lies on at least $k$ of the planes and none of the planes pass through the origin. 
\end{mydef}

\begin{thm}
    For any 3-dimensional generic half-grid $\Gamma$ of order $n+1$, 
    \[
\frac{31}{18}nk-O(k)
\leq
\cov{k}(\Gamma)
\leq
\frac{31}{18}nk+O(n^2+k),
\]
\end{thm}
\begin{proof}
    We give a construction covering the half-grid $k$ times using $\frac{31}{18}nk + O(n^2+k)$ planes. 
    First consider the following collection of planes:
    \begin{enumerate}
        \item\label{type 1} For each integer $1 \leq i \leq \frac{2n}{3}$, take $\lceil\frac{2n}{3}-i\rceil$ planes each of the form $x=a_i$, $y=b_i$, and $z=c_i$.
        \item\label{type 2} For each integer $1 \leq i \leq \frac{n}{3}$, take $i$ planes each of the form $\frac{x}{a_i}+\frac{y}{b_i}=1$, $\frac{y}{b_i}+\frac{z}{c_i}=1$, and $\frac{z}{c_i}+\frac{x}{a_i}=1$.
        \item\label{type 3} For each integer $\frac{n}{3} < i \leq n$, take $\lceil\frac{n}{3}\rceil$ planes each of the form $\frac{x}{a_i}+\frac{y}{b_i}=1$, $\frac{y}{b_i}+\frac{z}{c_i}=1$, and $\frac{z}{c_i}+\frac{x}{a_i}=1$.
        \item\label{type 4} For each integer $1 \leq i \leq \frac{2n}{3}$, take $\lceil|i-\frac{n}{3}|\rceil$ planes of the form $\frac{x}{a_i}+\frac{y}{b_i}+\frac{z}{c_i}=1$.
        \item\label{type 5} For each integer $ \frac{2n}{3} < i \leq n$, take $\lceil\frac{n}{3}\rceil$ planes of the form $\frac{x}{a_i}+\frac{y}{b_i}+\frac{z}{c_i}=1$.
    \end{enumerate}
    These are $\frac{31}{18}n^2+O(n)$ planes that don't pass through the origin. 

    We will prove that they cover every point except the origin at least $n$ times, and then repeat this construction to get the required upper bound on $\cov{k}({\Gamma})$.  
Let $(a_r, b_s, c_t)$ be any point in the half-grid.  
WLOG assume $r \geq s \geq t$, and then we have the following $9$ cases:  

\begin{enumerate}
    \item $0 < r, s, t \leq \frac{2n}{3}$: Planes of type \ref{type 1} cover it at least $\frac{2n}{3} - r + \frac{2n}{3} - s + \frac{2n}{3} - t \geq n$ times.  
    \item $0 < s, t \leq \frac{2n}{3}$ and $\frac{2n}{3} < r$: In this case, $s + t \leq n - r < \frac{n}{3}$. Planes of type \ref{type 1} cover it at least $\frac{2n}{3} - s + \frac{2n}{3} - t > n$ times.  
    \item $t = 0$, $0 < s \leq \frac{n}{3}$, and $\frac{2n}{3} < r$: In this case, planes of type \ref{type 1} cover it at least $\frac{2n}{3} - s$ times, planes of the form $\frac{y}{b_s} + \frac{z}{c_s} = 1$ from type \ref{type 2} cover it $s$ times, and finally planes of the form $\frac{z}{c_r} + \frac{x}{a_r} = 1$ from type \ref{type 3} cover it at least $\frac{n}{3}$ times, adding up to a total of at least $n$ times.  
    \item $t = 0$, $0 < s \leq \frac{n}{3}$, and $\frac{n}{3} < r \leq \frac{2n}{3}$: In this case, planes of type \ref{type 1} cover it at least $\frac{2n}{3} - r + \frac{2n}{3} - s = \frac{4n}{3} - r - s$ times, planes of the form $\frac{y}{b_s} + \frac{z}{c_s} = 1$ from type \ref{type 2} cover it $s$ times, and finally planes of the form $\frac{z}{c_r} + \frac{x}{a_r} = 1$ from type \ref{type 3} cover it at least $\frac{n}{3}$ times, for a total of at least $\frac{5n}{3} - r \geq n$ times.  
    \item $t = 0$ and $0 < r, s \leq \frac{n}{3}$: In this case, planes of type \ref{type 1} cover it at least $\frac{2n}{3} - r + \frac{2n}{3} - s = \frac{4n}{3} - r - s$ times, and planes of the form $\frac{z}{c_r} + \frac{x}{a_r} = 1$ and $\frac{y}{b_s} + \frac{z}{c_s} = 1$ from type \ref{type 2} cover it $r$ and $s$ times respectively, which gives at least $\frac{4n}{3} > n$ times in total.  
    \item $t = 0$ and $\frac{n}{3} < r, s$: In this case, $r + s \leq n$ implies $r, s < \frac{2n}{3}$. Hence planes of type \ref{type 1} cover it at least $\frac{2n}{3} - r + \frac{2n}{3} - s = \frac{4n}{3} - r - s$ times, and planes of the form $\frac{z}{c_r} + \frac{x}{a_r} = 1$ and $\frac{y}{b_s} + \frac{z}{c_s} = 1$ from type \ref{type 3} cover it at least $\frac{n}{3}$ times each, giving a total of at least $2n - r - s \geq n$ times.  
    \item $t = s = 0$, $0 < r \leq \frac{n}{3}$: In this case, planes of type \ref{type 1} cover it at least $\frac{2n}{3} - r$ times, planes of the form $\frac{z}{c_r} + \frac{x}{a_r} = 1$ and $\frac{x}{a_r} + \frac{y}{b_r} = 1$ from type \ref{type 2} cover it $r$ times each, and planes of type \ref{type 4}, of the form $\frac{x}{a_r}+\frac{y}{b_r}+\frac{z}{c_r}=1$, cover it at least $\frac{n}{3} - r$ times, for a total of at least $n$ times.  
    \item $t = s = 0$, $\frac{n}{3} < r \leq \frac{2n}{3}$: In this case, planes of type \ref{type 1} cover it at least $\frac{2n}{3} - r$ times, planes of the form $\frac{z}{c_r} + \frac{x}{a_r} = 1$ and $\frac{x}{a_r} + \frac{y}{b_r} = 1$ from type \ref{type 3} cover it at least $\frac{n}{3}$ times each, and planes of type \ref{type 4}, of the form $\frac{x}{a_r}+\frac{y}{b_r}+\frac{z}{c_r}=1$, cover it at least $r - \frac{n}{3}$ times, for a total of at least $n$ times.  
    \item $t = s = 0$, $\frac{2n}{3} < r$: In this case, planes of the form $\frac{z}{c_r} + \frac{x}{a_r} = 1$ and $\frac{x}{a_r} + \frac{y}{b_r} = 1$ from type \ref{type 3} cover it at least $\frac{n}{3}$ times each, and planes of type \ref{type 5}, of the form $\frac{x}{a_r}+\frac{y}{b_r}+\frac{z}{c_r}=1$, cover it at least $\frac{n}{3}$ times, for a total of at least $n$ times.
\end{enumerate}

    Now, we can write $k=qn+v$ for some $0 \leq v <n$, and apply the above construction $q$ times. For the remaining $v$, take the planes $x=a_i$, $y=b_i$, and $z=c_i$ each $v$ times for $1 \leq i \leq n$. This uses an additional $3vn=O(n^2)$ planes. Thus the total number of planes used is
    $$\frac{31}{18}n^2q+O(nq)+O(n^2)=\frac{31}{18}nk+O(n^2+k).$$
    Now we show that at least $\frac{31}{18}nk-O(k)=\frac{31}{18}nk-O(n^2+k)$ many planes are required. We again employ the dual LP, where points have to be given weights so that the sum of the weights on any plane is at most $k$, and we have to maximize the sum of the weights; this would give a lower bound for the number of planes required. \\

    Before we start, we note that it is sufficient to prove the bound for $3 \mid n$. This is because the generic half grid contains as its subsets generic half-grids of orders $n-1$ and $n-2$, and the lower bound for them differs from the lower bound for order $n$ by $O(k)$. We also assume $n$ is sufficienlty large, as any small $n$ can be dealt with by making the constant in $O(k)$ larger. \\

    We call the set of points of the form $(a_r,b_s,c_t)$ with $r+s+t=n$ the "front face" of the grid. We call the set of points on the front face with at least one coordinate $0$ the "face-boundary", and all other points on the front face the "interior of the front face". We know describe the weighting. For points on the axes, that is, with only one non-zero coordinate, we give them weight $\frac{k}{3}$. We give all face-boundary points weight $\frac{k}{6}$. 
    For the remaining points, non-zero weights are given only points in the interior of the front face, while all remaining points get weight $0$.
For these points we use a weighting inspired by \cite[Theorem 2.1]{basit2023covering}.
Namely, we give points of the form $(a_r,b_s,c_t)$ in the interior of the front face and $\max\{r,s,t\} \leq \frac{2n}{3}$ a weight of $\frac{3k}{2n(n+3)}(|r-\frac{n}{3}|+|s-\frac{n}{3}|+|t-\frac{n}{3}|)$, and give other points a weight of $0$. 
Note that this function is at most $3k/n$, which is strictly less than $k/6$ for $n$ large enough.

We now show that this weighting works. 
Let $H$ be a plane defined by an equation of the form $u^T x = 1$
Then there are three cases, depending on the size of the support of $u$. \\

\noindent
\textit{Case 1}: $|\mathrm{supp}(u)| = 3$. \\
Then $H$ intersects the grid in at most $3$ points. 
For $n$ large enough, 
the weight of every point is at most $k/3$, and hence the sum of weights is at most $k$. $\blacksquare$ \\

\noindent
\textit{Case 2}: $|\mathrm{supp}(u)| = 2$. \\
WLOG, say $\vec{u} = (u_1, u_2, 0)$ where $u_1u_2 \neq 0$. 
Since the grid is generic, $H$ intersects each $xy$-cross-section of the grid in at most $2$ points. Also, if a point $(a_r,b_s,c_t)$ is on $H$, then so is every point $(a_r,b_s,c_{t'})$ for all $0 \leq t' \leq n$. Thus the intersection of $H$ with the grid forms at most two vertical lines. 
Each of these lines has at most $1$ point on the $xy$ plane and one point on the front face, giving us a total weight of at most $2(k/3 + k/6) = k$. $\blacksquare$
\\

\noindent
\textit{Case 3}: $|\mathrm{supp}(u)| = 1$. \\
WLOG, say $\vec{u}=(u_1,0,0)$. Assuming $H$ contains at least one point from $\Gamma$, $u_1$ must be $\frac{1}{a_r}$ for some $r$. Therefore, the intersection with the top face is precisely the points $(a_r,b_s,c_{n-r-s})$ for $0 \leq s \leq n-r$. We call this set of points a "crooked line". The $s=0$ and $s=n-r$ points correspond to the two face-boundary points, so we focus on the interior first. We will prove that sum of weights assigned to these points is at most $\frac{k}{3}$. To make things easier, we extend the weighting of the interior to the face boundary too (so for case $rst=0$) and prove the bound, which would immediately imply the bound for just the interior points. \\

Let $n=3m$ and write $t=n-r-s$. We define
$$\rho(r,s,t)=\max\!\Big\{\Big|r-\tfrac{n}{3}\Big|,\;\Big|s-\tfrac{n}{3}\Big|,\;\Big|t-\tfrac{n}{3}\Big|\Big\} = \max\{|r-m|,\,|s-m|,\,|t-m|\}.$$
Note that since $r+s+t=n=3m$ on the front face, the three deviations $r-m$, $s-m$, $t-m$ sum to zero. It follows that
$$|r-m|+|s-m|+|t-m|=2\,\max\{|r-m|,\,|s-m|,\,|t-m|\}=2\,\rho(r,s,t),$$
so the max and the sum of absolute deviations from $n/3$ differ only by a factor of $2$.

If $r>2m$, then no point on the crooked
line receives positive weight, so the claim is immediate. Assume
$0\leq r\leq 2m$.

First suppose $0\leq r\leq m$. Set
$p=m-r$ and $u=s-m$. Then
$r-m=-p$ and $ n-r-s-m=p-u$.
The cutoff $\max\{r,s,n-r-s\}\leq 2m$ restricts $u$ to a subinterval of
$-(m-p)\leq u\leq m$.
Thus it is enough to sum over this whole interval. For this interval,
$$\rho(r,s,n-r-s)=\max\{p,|u|,|p-u|\}.$$
Splitting into the ranges
$$-(m-p)\leq u\leq 0,\qquad 1\leq u\leq p,\qquad p+1\leq u\leq m,$$
we get respectively
$$\rho(r,s,n-r-s)=p-u,\qquad \rho(r,s,n-r-s)=p,\qquad \rho(r,s,n-r-s)=u.$$
Therefore
$$\sum_s \rho(r,s,n-r-s) \leq \sum_{u=-(m-p)}^{0}(p-u) + \sum_{u=1}^{p}p + \sum_{u=p+1}^{m}u.$$
Computing these sums gives
$$\sum_s \rho(r,s,n-r-s) \leq (m-p+1)p+\frac{(m-p)(m-p+1)}2 +p^2 +\frac{m(m+1)}2-\frac{p(p+1)}2 = m(m+1).$$

Now suppose $m\leq r\leq 2m$. Set $p=r-m$ and $u=s-m$.
Then $r-m=p$ and  $n-r-s-m=-p-u$. The cutoff restricts $u$ to a subinterval of
$-m\leq u\leq m-p$. Thus it is enough to sum over this whole interval. For this interval,
$$\rho(r,s,n-r-s)=\max\{p,|u|,|p+u|\}.$$
Splitting into the ranges
$$-m\leq u\leq -p,\qquad -p+1\leq u\leq 0,\qquad 1\leq u\leq m-p,$$
we get respectively
$$\rho(r,s,n-r-s)=-u,\qquad \rho(r,s,n-r-s)=p,\qquad \rho(r,s,n-r-s)=p+u.$$
Therefore
$$\sum_s \rho(r,s,n-r-s) \leq \sum_{u=-m}^{-p}(-u) + \sum_{u=-p+1}^{0}p + \sum_{u=1}^{m-p}(p+u).$$
Computing these sums gives
$$\sum_s \rho(r,s,n-r-s) \leq \frac{m(m+1)}2-\frac{p(p-1)}2 +p^2 +p(m-p)+\frac{(m-p)(m-p+1)}2 = m(m+1).$$

Thus, for every fixed $r$,
$$\sum_s \rho(r,s,n-r-s)\leq m(m+1).$$
Consequently, the sum of weights is at most
$$\frac{k}{3m(m+1)} \sum_s \rho(r,s,n-r-s) \leq \frac{k}{3}.$$
Now, $H$ has, apart from the crooked line, two face boundary points and one point $(a_r,0,0)$ on the axis, which have total weight $\frac{k}{6}+2 \cdot\frac{k}{3}=\frac{2k}{3}$, so total sum of weights of points on $H$ is again at most $k$. $\blacksquare$
\\ 
    
Now we count the total weight assigned to all points. The sum of the weights on the axes and face-boundary is $nk+\frac{nk}{2}-O(k)$. We also estimate the total mass assigned by this weighting on the interior
of the top face. Recall that $n=3m$, and we re-use $\rho$ from Case 3 above.
For $1\leq q\leq m$, the set of points on the front face satisfying
\[
r+s+t=3m,\qquad \rho(r,s,t)=q
\]
forms a hexagonal shell of size $6q$. Therefore
\[
\sum_{\substack{r+s+t=3m\\ \rho(r,s,t)\leq m}} \rho(r,s,t)
=
\sum_{q=1}^{m} q\cdot 6q
=
6\sum_{q=1}^{m}q^2
=
m(m+1)(2m+1).
\]
This count includes the boundary points with exactly one coordinate equal to
$0$. Among those, the points satisfying $\rho(r,s,t)\leq m$ are precisely
the points on the three boundary segments $\{r=0 \text{ \ \ and \ \ } m\leq s,t\leq 2m\} $,
and the two analogous segments. Each such boundary segment has $m+1$ points,
and every one of these points has $\rho(r,s,t)=m$. Hence the total
contribution of $\rho$ from these boundary points is $3m(m+1)$.
Thus
\[
\sum_{\substack{r+s+t=3m\\ r,s,t>0\\ \rho(r,s,t)\leq m}}
\rho(r,s,t)
=
m(m+1)(2m+1)-3m(m+1)
=
m(m+1)(2m-2).
\]
Since the weight of each point is
$\frac{k}{3m(m+1)}\rho(r,s,t),$
the total weight assigned to the interior face points is
\[
\frac{k}{3m(m+1)}
\cdot
m(m+1)(2m-2)
=
\frac{k}{3}(2m-2).
\]
Using $n=3m$, this equals $\frac{2nk}{9}-\frac{2k}{3}$.
In particular, the total contribution of the interior front face weighting is $\frac{2nk}{9}-O(k)$. Adding this to the axes and face-boundary gives a total weight of $\frac{31}{18}nk-O(k)$, as required.
\end{proof}

\subsection{Structured grids}
\label{sec:structured}

In this section, we prove Theorem~\ref{thm:structured}, where the missing point of the half-grid does not need to be the vertex of the grid.

\begin{mydef}[Half-rectangular grids]
For positive integers $m \leq n$, an $m \times n$ half-rectangular grid is the set of all integer points $(x,y)$ such that $0 \leq x \leq n-1$, $0 \leq y \leq m-1$, and $(m-1)x+(n-1)y \leq (m-1)(n-1)$.
\end{mydef}

\begin{thm}
    For an $m \times n$ half-rectangular grid $\cH$, and point $P=(x_0,y_0) \in \cH$, $\cov{1}(\cH, P)=n-\lceil \frac{n-m}{m-1} y_0 \rceil -1$.
\end{thm}
\begin{proof}
    The idea is to find a large enough grid contained in $\cH$ that contains $P$, and apply the bound of $s - 1 + t - 1$ on grids of size $s \times t$. 
    Indeed, consider the grid $\{0 \leq x \leq n-\lceil \frac{n-1}{m-1}y_0\rceil-1\} \times \{0 \leq y \leq y_0\}$. This is clearly contained in $\cH$ and contains $P$ (in fact in the topmost row). Further, it has dimensions $(n-\lceil \frac{n-1}{m-1}y_0\rceil) \times (y_0+1)$. If we want to cover $\cH \setminus P$ without covering $P$, we have to at least cover the above grid without covering $P$. Therefore by the standard bound on grids, at least $(n-\lceil \frac{n-1}{m-1}y_0\rceil) + (y_0+1)-2=n-\lceil \frac{n-m}{m-1} y_0 \rceil -1$ lines are needed. \\
    
    Now we give the construction. Consider the following family of $n-\lceil \frac{n-m}{m-1} y_0 \rceil -1$ lines:
    \begin{enumerate}
            \item $y_0$ lines of the form   $y=i$ for $0 \leq i \leq y_0-1$.
            \item $x_0$ lines of the form $x=i$ for $0 \leq i \leq x_0-1$.
            \item $n-1-x_0-y_0-\lceil \frac{n-m}{m-1} y_0 \rceil$ lines of the form $x+y=i$ for $n-\lceil \frac{n-m}{m-1} y_0 \rceil-1 \geq i \geq x_0+y_0+1$.
    \end{enumerate}
    Clearly these lines don't contain $(x_0,y_0)$. Furthermore, any point with $y$-coordinate at most $y_0-1$, and any point with $x$-coordinate at most $x_0-1$, are covered by the first two sets of lines. 
    Now consider some other point $(x,y)$ with $x \geq x_0$ and $y \geq y_0$. Then, $$x+y \leq n-1- \frac{n-1}{m-1}y+y = n-1-\frac{n-m}{m-1}y \leq n-1-\frac{n-m}{m-1}y_0$$
    and since $x+y$ is an integer, $x+y \leq n-1-\lceil \frac{n-m}{m-1} y_0 \rceil$. Also, $x+y \geq x_0+y_0$ and equality holds iff $(x,y)=(x_0,y_0)$, so $x+y \geq x_0+y_0+1$ for all other points. Therefore all points in $\cH \setminus P$ are covered.
\end{proof}

\section*{Acknowledgments}
This work was carried out as part of the 2023 Polymath Jr. program, supported by NSF award DMS-2313292.
We thank the organizers and participants of this program. 
We thank Lander Verlinde for helpful comments on our first draft.
We also thank the anonymous referees for a thorough reading of the manuscript and their detailed remarks.

\end{document}